\documentclass[11pt,a4paper,reqno]{amsart}

\usepackage{amsfonts}
\usepackage{amsmath}
\usepackage{amssymb}
\usepackage{amsthm}
\usepackage{booktabs}
\usepackage{bbm}
\usepackage{bm}
\usepackage{caption}
\usepackage{dsfont}
\usepackage{fancyhdr}
\usepackage{float}
\usepackage[letterpaper,top=2.5cm,bottom=2.5cm,left=2.5cm,right=2.5cm,marginparwidth=1.75cm]{geometry}
\usepackage{graphicx}
\usepackage{hyperref}
\usepackage{cleveref}
\usepackage{mathrsfs}
\usepackage{multirow}
\usepackage{natbib}
\usepackage{subfigure}
\usepackage{yfonts}
\usepackage[pagewise]{lineno}\linenumbers

\DeclareMathAlphabet{\mathpzc}{OT1}{pzc}{m}{it}
\makeatletter

\def\BC{\mathbb C}
\def\BN{\mathbb N}
\def\BR{\mathbb R}

\def\cF{\mathcal F}
\def\rRe{\mathrm{Re}}
\def\T{\mathrm T}
\def\rd{\mathrm d}
\def\e{\mathrm e}

\def\Ga{\Gamma}

\def\Om{\Omega}
\def\al{\alpha}
\def\be{\beta}
\def\ga{\gamma}
\def\de{\delta}

\def\ka{\kappa}
\def\la{\lambda}
\def\si{\sigma}
\def\vp{\varphi}
\def\om{\omega}
\def\f{\frac}
\def\nb{\nabla}
\def\ov{\overline}
\def\pa{\partial}
\def\tri{\triangle}

\newtheorem{thm}{Theorem}[section]

\newtheorem{ex}{Example}

\newtheorem{prob}[thm]{Problem}

\newtheorem{lem}[thm]{Lemma}

\title[Recovery of Multiple Parameters in Nonlocal Diffusion Equations]{Simultaneous Recovery of Multiple Parameters in Nonlocal Diffusion Equations from Internal Measurements}

\author[K. Yu]{Kai Yu}
\address{School of Mathematics and Statistics, Ningbo University, Ningbo 315211, China.}
\email{yukaimailbox@163.com}

\author[Z. Li]{Zhiyuan Li}
\address{School of Mathematics and Statistics, Ningbo University, Ningbo 315211, China.}
\email{lizhiyuan@nbu.edu.cn}

\author[Y. Liu]{Yikan Liu}\thanks{Corresponding author: Yikan Liu.}
\address{Department of Mathematics, Kyoto University, Kitashirakawa-Oiwakecho, Sakyo-ku, Kyoto 606-8502, Japan.}
\email{liu.yikan.8z@kyoto-u.ac.jp}

\date{}

\begin{document}
\nolinenumbers
\maketitle

\begin{abstract}
This paper is devoted to simultaneously recovering multiple parameters from internal measurements for nonlocal diffusion equations. The uniqueness of the inverse problem is established by employing the asymptotic behavior of solutions, analytic continuation, the Laplace transform, and properties of analytic functions. For numerical reconstruction, we apply the Levenberg-Marquardt method to obtain a stable approximate solution of the inverse problem. Numerical examples are provided to demonstrate the efficiency of the proposed algorithm and to validate our theoretical findings.
\end{abstract}


\section{Introduction}

In recent years, nonlocal diffusion equations have become powerful mathematical tools with extensive applications across physics, chemistry, biology, and biochemistry \cite{2000The,2002Waiting-times,2002Subdiffusion-limited}. These equations effectively capture complex phenomena through their inherent memory effects, nonlocal interactions, and anomalous diffusion characteristics. The temporal fractional derivative models particle sticking and trapping behaviors, whereas the spatial fractional derivative describes long-range particle jumps. The interplay between these two components produces concentration profiles with sharper peaks and heavier tails, clearly distinguishing them from classical diffusion profiles. Moreover, space-time fractional diffusion equations offer a more precise framework than purely time-fractional diffusion equations for describing certain physical processes, particularly those involving coupled memory and jump dynamics.

Let $\Om$ be a bounded domain in $\BR^d$ with a sufficiently smooth boundary $\pa\Om$ and $T>0$ be a fixed final time. In this paper, we consider the following nonlocal diffusion problem
\begin{equation}\label{eq1.1}
\begin{cases}
\pa_t^\al(u(x,t)-a(x))+(-\tri)^\be u(x,t)=F(x,t), & (x,t)\in\Om\times(0,T),\\
u(x,t)=a(x), & (x,t)\in\ov\Om\times\{0\},\\
u(x,t)=0, & (x,t)\in\pa\Om\times(0,T),
\end{cases}
\end{equation}
where $\al,\be\in(0,1)$ are constants. Here $\pa_t^\al$ denotes the Caputo derivative of order $\al$ with respect to $t$ defined by
\[
\pa_t^\al u(x,t):=\f1{\Ga(1-\al)}\int_0^t\f{\pa u(x,\tau)}{\pa\tau}\f{\rd\tau}{(t-\tau)^\al},\quad t>0.
\]
The Riemann-Liouville integral operator of order $1-\al$ is given by the formula
\[
I_t^{1-\al}u(x,t):=\int_0^t\f{(t-\tau)^{-\al}}{\Ga(1-\al)}u(x,\tau)\,\rd\tau.
\]
We also define the integral with an order $0$ as the identity operator, i.e., $I_t^0u=u$. The Riemann-Liouville fractional derivative of the order $\al>0$ is defined by
\[
D_t^\al u(t)=\f\rd{\rd t}\int_0^t\f{(t-\tau)^{-\al}}{\Ga(1-\al)}u(\tau)\,\rd\tau,
\]
where $\Ga$ is the Gamma function.

The fractional Laplacian $(-\tri)^\beta$ is defined via the spectral decomposition of the Laplace operator; a detailed summary of the definition can be found in Section \ref{sec2}, and we refer to \cite{2016Simultaneous,2015An inverse} for further background. Note that as $\alpha,\beta\to 1$, the Caputo derivative $\pa_t^\alpha u$ tends to the first-order derivative $u_t$ and $(-\tri)^\beta$ tends to $-\tri$; consequently, model \eqref{eq1.1} reduces to the standard diffusion equation.

Now we assume that the source term $F$ in \eqref{eq1.1} is of the separated form
\[
F(x,t)=p(t)f(x),
\]
where $p(t)$ and $f(x)$ stand for the temporal and the spatial component of the source term, respectively. Our main objective is to prove uniqueness in the simultaneous identification of multiple parameters for \eqref{eq1.1} from internal measurements. We consider the following inverse problem.

\begin{prob}\label{prob-isp}
Let $p(t)$ be given and $\om$ be a nonempty subdomain of $\Om$. Under certain assumptions, can we uniquely determine $\al,$ $\be,$ $f(x)$ and $a(x)$ simultaneously by the partial interior observation data of $u$ in $\om\times(0,T)?$
\end{prob}

Inverse problems for nonlocal diffusion models have been intensively studied over the past decades, with particular emphasis on time-fractional diffusion equations due to their ability to capture memory effects and anomalous transport phenomena in a wide range of scientific and engineering applications. Since the literature on time-fractional diffusion equations is vast, we only highlight a few representative works here. In the context of single-parameter inversion, substantial contributions have been made; see, e.g., \cite{2019Inverse problems li,2019Inverse li,2019Inverse liu,2011Initial}. More recently, the focus has shifted toward multi-parameter inversion, which better reflects the complexity of real-world systems; see, e.g., \cite{2023RecoveryCen,2021The,2022Uniqueness,2021Simultaneous}.

As mentioned above, nonlocal diffusion equations with time-fractional derivatives have been widely studied. However, to the best of the authors’ knowledge, inverse problems for nonlocal diffusion equations involving both space- and time-fractional derivatives have received relatively little attention. Within the existing studies on this topic, single-parameter inversions, such as recovering fractional orders or diffusion coefficients, have been investigated rather extensively. In particular, the inverse source problem, which aims to identify an unknown source term from boundary or interior measurements, has been widely examined; see, e.g., \cite{2018Inverse,2017Harnack's,2018An,2015An inverse}. Other parameter identification problems can be found in \cite{2020Inverse problems,2018Backward,2020The Calderon,2021Direct}. Beyond single-parameter inversion, recent research has moved toward more complex scenarios involving multi-parameter inversion, where several unknown quantities must be determined simultaneously. Notably, Tatar et al. \cite{2016Simultaneous} were the first to simultaneously identify the time- and space-fractional orders in a space-time fractional diffusion equation, while Guerngar et al. \cite{2021SimultaneousGuerngar} also recovered the fractional exponents for such equations. Zhang et al. \cite{2018Bayesian} adopted a Bayesian approach for inverse problems in time-space fractional diffusion equations. Janno \cite{2020Determination} extended the investigation to time-dependent sources and parameters in nonlocal diffusion and wave equations using final observation data, while Malik et al. \cite{2021SimultaneousMalik} developed methods for concurrently determining source terms and diffusion concentrations in multi-term space-time fractional diffusion systems. Moreover, several efficient numerical reconstruction algorithms have been proposed; see, e.g., \cite{2024Recovering,2023The quasi-reversibility,2021Unknown}. Despite these advances, research on multi-parameter inversion for space-time fractional equations remains far from sufficient. This gap motivates the present work.

This paper seeks to contribute to the understanding of multi-parameter inversion for space-time fractional diffusion equations by establishing a uniqueness result and a reconstruction algorithm. Specifically, we focus on simultaneously recovering the spatial component of the source term, the initial value, and the fractional orders in space and time from internal measurements. The inverse problem is solved numerically using the Levenberg-Marquardt method.

The remainder of the paper is organized as follows. Section \ref{sec2} recalls some preliminaries needed for the subsequent analysis. In Section \ref{sec3}, we prove the uniqueness of the inverse problem. Section \ref{sec4} presents the Levenberg-Marquardt method for numerical reconstruction, and Section \ref{sec5} reports numerical results to illustrate its performance. Finally, Section \ref{sec6} concludes the paper.


\section{Preliminary and the main result}\label{sec2}

In this section, we first set up notations and the terminology, and review some standard facts on the fractional calculus. Let $L^2(\Om)$ be a usual $L^2$-space with the inner  product $(\,\cdot\,,\,\cdot\,)$ and $H^1_0(\Om)$, $H^2(\Om)$, etc.\! denote the usual Sobolev spaces. By ${H^\al(0,T)}$, we denote the fractional Sobolev space with the norm
\[
\|\psi\|_{H^\al(0,T)}=\left(\|\psi\|_{L^2(0,T)}^2+\int_0^T\!\!\!\int_0^T\f{|\psi(t)-\psi(\tau)|^2}{|t-\tau|^{1+2\al}}\,\rd\tau\rd t\right)^{1/2}
\]
(e.g., Adams \cite{AF03}). Next, $H_\al(0,T)$ is a subspace of $H^\al(0,T)$ defined by
\[
H_\al(0,T):=\left\{\begin{alignedat}{2}
& \{\psi\in H^\al(0,T)\mid\psi(0)=0\}, &\quad & 1/2<\al<1,\\
&\left\{\psi\in H^{1/2}(0,T)\mid\int_0^T\f{|\psi(t)|^2}t\,\rd t<\infty\right\}, &\quad & \al=1/2,\\
& H^\al(0,T), &\quad & 0<\al<1/2,
\end{alignedat}\right.
\]
equipped with the norm
\[
\|\psi\|_{H_\al(0,T)}:=\left\{\begin{alignedat}{2}
& \|\psi\|_{H^\al(0,T)}, &\quad & 0<\al<1,\ \al\ne1/2,\\
&\left(\|\psi\|_{H^{1/2}(0,T)}^2+\int_0^T\f{|\psi(t)|^2}t\,\rd t\right)^{1/2}, &\quad & \al=1/2
\end{alignedat}\right.
\]
(see e.g.\! Gorenflo, Luchko and Yamamoto \cite{GL15}).

In the spatial direction, we equip $-\tri$ with the homogeneous Dirichlet boundary condition and specify its domain as $D(-\tri)=H^2(\Om)\cap H_0^1(\Om)$. We denote the eigensystem of the operator $-\tri$ with homogeneous Dirichlet boundary condition as $\{(\la_n,\vp_n)\}_{n=1}^\infty$, namely
\[
\begin{cases}
-\tri\vp_n=\la_n\vp_n & \mbox{in }\Om,\\
\vp_n=0 &\mbox{on }\pa\Om.
\end{cases}
\]
It is known that $0<\la_1<\la_2\le\cdots$ and $\{\vp_n\}_{n=1}^\infty$ forms a complete orthonormal basis of $L^2(\Om)$. Then we can define the fractional Sobolev space $D((-\tri)^\ga)$ and the corresponding fractional Laplacian operator $(-\tri)^\ga$ with $\ga>0$ as 
\begin{gather*}
D((-\tri)^\ga):=\left\{\psi\in L^2(\Om)\mid\|\psi\|_{D((-\tri)^\ga)}:=\|(-\tri)^\ga\psi\|_{L^2(\Om)}<\infty\right\},\\
(-\tri)^\ga\psi:=\sum_{n=1}^\infty\la_n^\ga(\psi,\vp_n)\vp_n.
\end{gather*}

Next, we invoke the Mittag-Leffler function
\[
E_{\al,\be}(z):=\sum_{k=0}^\infty\f{z^k}{\Ga(\al k+\be)},\quad z\in\BC,\ \al>0,\be\in\BR
\]
and recall the following frequently used fact.

\begin{lem}[see \cite{1998Fractional}]\label{lem2.13}
If $\al,\be\in(0,1),$ then $0<E_{\al,\be}(-\eta)\le1$ for any $\eta\ge0$.
\end{lem}

Let us give the well-posedness of a weak solution to the direct problem \eqref{eq1.1}.

\begin{lem}\label{lem3.1}
Let $a\in D((-\tri)^{1/2}),$ $f\in L^2(\Om)$ and $p\in L^\infty(0,T)$. Then there exists a unique solution to \eqref{eq1.1} such that $u\in L^2(0,T;D((-\tri)^\be))$ and $u-a\in H_\al(0,T;L^2(\Om))$. Moreover, the explicit solution is given by
\begin{equation}\label{eq3.1}
u(x,t)=\sum_{n=1}^\infty E_{\al,1}(-\la_n^\be t^\al)\,a_n\vp_n(x)+\sum_{n=1}^\infty\int_0^t\tau^{\al-1}E_{\al,\al}(-\la_n^\be\tau^\al)p(t-\tau)\,\rd\tau\,f_n\vp_n(x),
\end{equation}
where $a_n:=(a,\vp_n)$ and $f_n:=(f,\vp_n)$. Further, there exists a constant $C>0$ depending only on $\al,\be,T,\Om$ such that
\[
\|u\|_{L^2(0,T;D((-\tri)^\be))}+\|u-a\|_{H_\al(0,T;L^2(\Om))}\le C\left(\|a\|_{H^1(\Om)}+\|p\|_{L^\infty(0,T)}\|f\|_{L^2(\Om)}\right).
\]
\end{lem}	

The proof of Lemma \ref{lem3.1} follows the same line as that in \cite{2025Inverse source}, and hence we omit the details here.

Now we state the main theorem concerning the uniqueness of Problem \ref{prob-isp}.

\begin{thm}\label{thm1.1}
Assume that $a_i\not\equiv0$ are real analytic on $\ov\Om,$ and satisfy
\[
(-\tri)^j a_i=0\quad\text{on }\partial\Om,\quad j=0,1,\dots,\quad i=1,2,
\]
and let $\al_i,\be_i\in(0,1),$ $f_i\in D((-\tri)^{1/2})$ for $i=1,2$, \(p\not\equiv 0\) and $p\in C[0,T]$ with $p(0)=0$. Suppose that $u_i$ is the solution to \eqref{eq1.1} with $\al=\al_i,$ $\be=\be_i,$ $a=a_i$ and $f=f_i$ for $i=1,2$. Then $u_1=u_2$ in $\om\times(0,T)$ implies $(\al_1,\be_1,a_1,f_1)=(\al_2,\be_2,a_2,f_2)$.
\end{thm}


\section{Proof of the main result}\label{sec3}

In the sequel, we shall use $C$, with or without subscript, to denote a generic positive constant, which may take different values line by line. Before we start proving the theorem, we prepare the following lemmas.

\begin{lem}\label{lem 4.1}
Let $a\in D(-\tri),$ $f\in D((-\tri)^{1/2})$ and $p\in C[0,T]$. Then we have the following asymptotic expansion
\[
u(x,t)=a(x)-\f{t^\al}{\Ga(\al+1)}\left((-\tri)^\be a(x)-p(0)f(x)\right)+o\left(t^\al\right),\quad x\in\Om\mbox{ as }t\to0.
\]
\end{lem}

\begin{proof}
By the definition of the Mittag-Leffler function, we can directly calculate
\[
E_{\al,1}(-\la_n^\be t^\al)=1-\f{\la_n^\be t^\al}{\Ga(\al+1)}+t^{2\al}\la_n^{{2\be}}E_{\al,2\al+1}(-\la_n^\be t^\al).
\]
Applying the mean value theorem for integrals, we have
\begin{align*}
\int_0^t p(\tau)(t-\tau)^{\al-1}E_{\al,\al}(-\la_n^\be(t-\tau)^\al)\,\rd\tau & =\f{p(\xi(t;\al))}{\la_n^\be}\left(1-E_{\al,1}(-\la_n^\be t^\al)\right)\\
& =p(\xi(t;\al))\left(\f{t^\al}{\Ga(\al+1)}-t^{2\al}\la_n^\be E_{\al,2\al+1}(-\la_n^\be t^\al)\right),
\end{align*}
where $0<\xi(t;\al)<t$. Then it follows that
\begin{equation}\label{4.4}
p(\xi(t;\al))=p(0)+o(1) \quad\text{ as } t\to 0.
\end{equation}
Substituting the above equalities into \eqref{eq3.1}, we arrive at
\begin{align*}
u(x,t) & =\sum_{n=1}^\infty a_n\vp_n(x)-\f{t^\al}{\Ga(\al+1)}\sum_{n=1}^\infty\la_n^\be a_n\vp_n(x)\\
&\quad\,+t^{2\al}\sum_{n=1}^\infty\la_n^{{2\be}}E_{\al,2\al+1}(-\la_n^\be t^\al)a_n\vp_n(x)+\f{t^\al p(\xi(t;\al))}{\Ga(\al+1)}\sum_{n=1}^\infty f_n\vp_n(x)\\
&\quad\,-t^{2\al}p(\xi(t;\al))\sum_{n=1}^\infty\la_n^\be E_{\al,2\al+1}(-\la_n^\be t^\al) f_n\vp_n(x)\\
& =:\sum_{k=1}^{5}I_k. 
\end{align*}
In the $L^2(\Om)$-sense, it follows immediately from Lemma \ref{lem2.13} that
\begin{gather*}
I_1=a(x),\quad I_2=-\f{t^\al}{\Ga(\al+1)}(-\tri)^\be a(x),\quad|I_3|\le t^{2\al}\|a\|_{H^2(\Om)},\\
I_4=\f{t^\al p(\xi(t;\al))}{\Ga(\al+1)}f(x),\quad|I_5|\le C t^{2\al}p(\xi(t;\al))\|f\|_{H^1(\Om)}. 
\end{gather*}
Passing $t\to0$, we can employ \eqref{4.4} to conclude the desired result.
\end{proof}

\begin{lem}\label{lem3.2}
If $a$ is real analytic in $\ov\Om$ and satisfies the corresponding Dirichlet compatibility conditions 
\begin{equation}\label{eq com_con}
(-\tri)^j a=0\quad\mbox{on }\pa\Om,\quad j=0,1,\dots,    
\end{equation}
then the function
\[
(-\tri)^\be a=\sum_{n=1}^\infty\la_n^\be a_n\vp_n,\quad a_n:=(a,\vp_n)_{L^2(\Om)},
\]
is real analytic in $\Om$.
\end{lem}

\begin{proof}
Since $a$ is real analytic on $\ov\Om$ and any analytic function belongs to the Gevrey class $G^1$ (see \cite[page 19]{1993Linear pa}), there exist constants $C,M>0$ such that for every multi-index $k:=(k_1,\dots, k_d)\in\{0,1,2,\dots\}^d$, 
\begin{equation}\label{eq anal_esti}
|\pa^k a(x)|\le C M^{|k|}k!,\quad  \forall\,x\in \ov\Om.    
\end{equation}
Therefore, 
\[
\|\pa^k a\|_{L^\infty(\Om)}=\sup_{x\in\Om}|\pa^k a(x)|\le C M^{|k|} k!.
\]
Moreover, for every \(m\in\BN\),
\[
(-\tri)^m=(-1)^m\left(\sum_{j=1}^d\pa_{x_j}^2\right)^m=(-1)^m\sum_{|k|=m}\f{m!}{k!}\pa^{2k}.
\]
Then
\[
\|(-\tri)^m a\|_{L^2(\Om)}\le\sum_{|k|=m}\f{m!}{k!}\|\pa^{2k}a\|_{L^2(\Om)}\le|\Om|^{1/2}\sum_{|k|=m}\f{m!}{k!}\|\pa^{2k}a\|_{L^\infty(\Om)}.
\]
By the estimate \eqref{eq anal_esti}, and
\[
(2k)!=\prod_{j=1}^d(2k_j)!\le(2m)!,
\]
we get
\[
\|(-\tri)^m a\|_{L^2(\Om)}\le C|\Om|^{1/2}M^{2m}(2m)!\sum_{|k|=m}\f{m!}{k!}.
\]
Moreover, \(\sum_{|k|=m}\f{m!}{k!}=d^m.\) Therefore
\[
\|(-\tri)^m a\|_{L^2(\Om)}\le C|\Om|^{1/2}(d^{1/2}M)^{2m}(2m)!\le C M^{2m}(2m)!,\quad m=1,2,\dots .
\]
In addition, the compatibility conditions \eqref{eq com_con} imply that \(a\in D((-\tri)^m)\) for every \(m\in\BN\). Therefore, by the spectral theorem,
\[
\sum_{n=1}^\infty\la_n^{2m}|a_n|^2=\|(-\tri)^m a\|_{L^2(\Om)}^2\le C M^{4m}((2m)!)^2.
\]
Consequently, for every \(m,n\in\BN\),
\[
|a_n|\le C M^{2m}(2m)!\la_n^{-m}.
\]
Choose \(m=\lceil\f{\sqrt{\la_n}}{2M}\rceil\). By Stirling's formula, there exists a constant $C>0$ such that
\[
(2m)!\le C(2m)^{2m+1/2}\,\e^{-2m}.
\]
Hence
\[
|a_n|\le C\sqrt{2m}\left(\f{2m M}{\e\sqrt{\la_n}}\right)^{2m}.
\]
Since $m\sim \sqrt{\la_n}/(2M)$ as $n\to\infty$, the last expression is bounded by \(C\exp(-c_0\sqrt{\la_n})\) for some constants $C,c_0>0$ independent of $n$. Thus
\begin{equation}\label{eq an_decay}
|a_n|\le C\exp\left(-c_0\sqrt{\la_n}\right),\quad n=1,2,\dots.
\end{equation}

Let
\[
b:=(-\tri)^\be a=\sum_{n=1}^\infty b_n\vp_n,\quad b_n:=\la_n^\be a_n .
\]
By the exponential decay estimate \eqref{eq an_decay} and Weyl's asymptotic law \(\la_n\sim Cn^{2/d}\) as $n\to\infty$, we have, for every fixed $0<\rho<c_0$,
\[
\sum_{n=1}^\infty\e^{2\rho\sqrt{\la_n}}|b_n|^2=\sum_{n=1}^\infty\e^{2\rho\sqrt{\la_n}}\la_n^{2\be}a_n^2\le C\sum_{n=1}^\infty\la_n^{2\be}\,\e^{-2(c_0-\rho)\sqrt{\la_n}}<\infty .
\]
Indeed, the exponential decay in $\sqrt{\la_n}$ dominates the polynomial factor
$\la_n^{2\be}$.

Define
\[
g_n:=\e^{\rho\sqrt{\la_n}}b_n,\quad n=1,2,\dots .
\]
Then $\{g_n\}_{n=1}^\infty\in\ell^2$. Hence the series
\[
g:=\sum_{n=1}^\infty g_n\vp_n
\]
converges in $L^2(\Om)$ and defines a function $g\in L^2(\Om)$. Moreover, by construction,
\[
b_n=\e^{-\rho\sqrt{\la_n}}g_n,
\]
and therefore
\begin{equation}\label{eq u_poisson}
b=\sum_{n=1}^\infty\e^{-\rho\sqrt{\la_n}}g_n\vp_n\quad\mbox{in }L^2(\Om).
\end{equation}
For $y>0$, define
\[
U(x,y):=\sum_{n=1}^\infty\e^{-y\sqrt{\la_n}}g_n\vp_n(x).
\]
For each fixed $y>0$, this series converges in $L^2(\Om)$, thus
\[
\|U(\,\cdot\,,y)\|_{L^2(\Om)}^2=\sum_{n=1}^\infty\e^{-2y\sqrt{\la_n}}|g_n|^2\le\sum_{n=1}^\infty|g_n|^2=\|g\|_{L^2(\Om)}^2 .
\]
For $N\in\BN$, define
\[
U_N(x,y):=\sum_{n=1}^N\e^{-y\sqrt{\la_n}}g_n\vp_n(x).
\]
Each eigenfunction $\vp_n$ belongs to $C^\infty(\Om)$ by interior elliptic regularity, and the function $y\longmapsto\e^{-y\sqrt{\la_n}}$ belongs to $C^\infty(0,\infty)$. Since $U_N$ is a finite sum, we have \(U_N\in C^\infty(\Om\times(0,\infty))\).

Moreover, using \(-\tri_x\vp_n=\la_n\vp_n,\) we compute
\[
\pa_y^2\left(\e^{-y\sqrt{\la_n}}g_n\vp_n(x)\right)=\la_n\,\e^{-y\sqrt{\la_n}}g_n\vp_n(x),
\]
whereas
\[
\tri_x\left(\e^{-y\sqrt{\la_n}}g_n\vp_n(x)\right)=\e^{-y\sqrt{\la_n}}g_n\tri_x\vp_n(x)=-\la_n\,\e^{-y\sqrt{\la_n}}g_n\vp_n(x).
\]
Therefore each summand satisfies
\[
\pa_y^2\left(\e^{-y\sqrt{\la_n}}g_n\vp_n(x)\right)+\tri_x\left(\e^{-y\sqrt{\la_n}}g_n\vp_n(x)\right)=0.
\]
Since the sum defining $U_N$ is finite, we may differentiate term by term and obtain
\begin{equation}\label{eq UN_harmonic}
\pa_y^2U_N+\tri_xU_N=0\quad\mbox{in }\Om\times(0,\infty).
\end{equation}

Let \(K\Subset\Om\times(0,\infty)\), and there exist an open set \(\Om_0\Subset\Om\) and numbers \(0<y_0<y_1<\infty\) such that \(K\subset\Om_0\times [y_0,y_1]\). For \(N>M\), Parseval's identity gives, for every \(y\in[y_0,y_1]\),
\[
\|U_N(\,\cdot\,,y)-U_M(\,\cdot\,,y)\|_{L^2(\Om)}^2=\sum_{n=M+1}^N\e^{-2y\sqrt{\la_n}} |g_n|^2\le\sum_{n=M+1}^N |g_n|^2 .
\]
Therefore,
\[
\|U_N-U_M\|_{L^2(\Om_0\times[y_0,y_1])}^2\le\int_{y_0}^{y_1}\|U_N(\,\cdot\,,y)-U_M(\,\cdot\,,y)\|_{L^2(\Om)}^2\,\rd y\le(y_1-y_0)\sum_{n=M+1}^N|g_n|^2 .
\]
According to \(\{g_n\}_{n=1}^\infty\in \ell^2\), the right-hand side tends to zero as \(M,N\to\infty\). Hence \(\{U_N\}\) is Cauchy in \(L^2(\Om_0\times[y_0,y_1])\). Its limit is precisely \(U\), by the definition of the series. Thus
\[
U_N\longrightarrow U\quad\mbox{in }L^2_{\mathrm{loc}}(\Om\times(0,\infty))\mbox{ as }N\to\infty.
\]

Let \(\Phi\in C_c^\infty(\Om\times(0,\infty))\). From \eqref{eq UN_harmonic}, we have
\[
\int_{\Om\times(0,\infty)}U_N(x,y)(\pa_y^2+\tri_x)\Phi(x,y)\,\rd x\rd y=0.
\]
Since \(\Phi\) has compact support and \(U_N\longrightarrow U\) in \(L^2_{\mathrm{loc}}(\Om\times(0,\infty))\), we may pass to the limit as \(N\to\infty\). This gives
\[
\int_{\Om\times(0,\infty)} U(x,y)(\pa_y^2+\tri_x)\Phi(x,y)\,\rd x\rd y=0.
\]
Therefore
\begin{equation}\label{eq U_distribution_harmonic}
(\pa_y^2+\tri_x)U=0\quad\mbox{in }\mathcal D'(\Om\times(0,\infty)).
\end{equation}
Since \(\pa_y^2+\tri_x=\tri_{x,y},\) where \(\tri_{x,y}\) denotes the usual Laplacian in the variables \((x,y)\in\BR^{d+1}\), equation \eqref{eq U_distribution_harmonic} means that \(U\) is harmonic in \(\Om\times(0,\infty)\) in the distributional sense. The analytic regularity theorem for the Laplacian (see \cite[Theorem 2.27]{1995Introduction}) implies that \(U\) is real analytic in \(\Om\times(0,\infty)\). Moreover, \(b(x)=U(x,\rho)\) for almost every \(x\in\Om\) by \eqref{eq u_poisson}, and the slice \(x\mapsto U(x,\rho)\) is real analytic in \(\Om\), it follows that \(b\) admits a real analytic representative in \(\Om\). Equivalently, after modifying \(b\) on a set of measure zero, \(b\) is real analytic in \(\Om\).
\end{proof}

Now we are in a position to proceed to the proof of the main result.

\begin{proof}[Proof of Theorem $\ref{thm1.1}$]
We first prove the uniqueness of $a$. According to Lemma \ref{lem 4.1} and $p(0)=0$, the coincidence $u_1=u_2$ in $\om\times(0,T)$ gives 
\begin{equation}\label{eq-asymp}
a_1(x)-\f{t^{\al_1}}{\Ga(\al_1+1)}(-\tri)^{\be_1}a_1(x)+o(t^{\al_1})=a_2(x)-\f{t^{\al_2}}{\Ga(\al_2+1)}(-\tri)^{\be_2}a_2(x)+o(t^{\al_2})
\end{equation}
for $x\in\om$ as $t\to0$. Then obviously $a_1=a_2$ in $\om$. Since $a_1$ and $a_2$ are analytic in $\ov\Om$, we simply conclude $a_1=a_2$ in $\ov\Om$ by the uniqueness theorem of analytic functions.

Now we prove $\al_1=\al_2$ by contradiction. Without loss of generality, let us assume $\al_1>\al_2$. Dividing both sides of \eqref{eq-asymp} by $t^{\al_2}$ and passing $t\to0$, we obtain
\[
\f1{\Ga(\al_2+1)}(-\tri)^{\be_2}a=0\quad\mbox{in }\om.
\]
Owing to the analyticity of $(-\tri)^{\be_2}a$ guaranteed by Lemma \ref{lem3.2}, it turns out that $(-\tri)^{\be_2}a\equiv0$ in $\Om$, which contradicts with the assumption that $a\not\equiv0$ is analytic on $\ov\Om$. Therefore, we conclude $\al_1=\al_2$.

Next, we prove the uniqueness of $\be$. Now that $a_1=a_2$ and $\al_1=\al_2$, dividing both sides of \eqref{eq-asymp} by $t^{\al_2}$ and passing $t\to0$ yield
\[
(-\tri)^{\be_1}a=(-\tri)^{\be_2}a\quad\mbox{in }\om,
\]
which, again by Lemma \ref{lem3.2}, indicates
\[
(-\tri)^{\be_1}a=(-\tri)^{\be_2}a\quad\mbox{in }\Om.
\]
Invoking the eigensystem $(\la_n,\vp_n)$ of $-\tri$, we rewrite above as
\[
\sum_{n=1}^\infty\left(\la_n^{\be_1}-\la_n^{\be_2}\right)a_n\vp_n=0\quad\mbox{in }\Om.
\]
Then the orthogonality of $\{\vp_n\}$ implies $(\la_n^{\be_1}-\la_n^{\be_2})a_n=0$ for all $n\in\BN$. Since it was assumed that $a\not\equiv0$ in $\Om$, there exists at least one $n_0\in\BN$ such that $a_{n_0}\ne0$. Therefore, we conclude $\be_1=\be_2$.

Finally, we deduce the uniqueness of the spatial source $f(x)$. Denoting $f_{1,n}=(f_1,\vp_n)$ and $f_{2,n}=(f_2,\vp_n)$, we employ the explicit solution \eqref{eq3.1} in Lemma \ref{lem3.1} to express
\[
\int_0^t p(\tau)\sum_{n=1}^\infty(t-\tau)^{\al-1}E_{\al,\al}\left(-\la_n^\be(t-\tau)^\al\right)(f_{1,n}-f_{2,n})\vp_n\,\rd\tau=0\quad\mbox{in }\om\times(0,T).
\]	
Using Titchmarsh's convolution theorem, it follows from the conditions on $p$ that
\begin{equation}\label{4.19}
\sum_{n=1}^\infty t^{\al-1}E_{\al,\al}\left(-\la_n^\be t^\al\right)f_{1,n}\vp_n=\sum_{n=1}^\infty t^{\al-1}E_{\al,\al}\left(-\la_n^\be t^\al\right)f_{2,n}\vp_n
\end{equation}
in $\om\times(0,T_1)$ for some constant $T_1\in(0,T)$.

On the one hand, for $z\in S_1:=\{z\in\BC\mid|\arg z|\le\min\{\f\pi\al-\f\pi2,\pi\},|z|\ge t_0\}$, it follows from Lemma \ref{lem2.13} that
\begin{equation}\label{4.20}
\left|\sum_{n=1}^\infty z^{\al-1}E_{\al,\al}\left(-\la_n^\be z^\al\right)f_n\vp_n(x)\right|\le C\,t_0^{-1}\|f\|_{L^2(\Om)},\quad\forall\,x\in\ov\Om.
\end{equation}
In other words, the above series is uniformly convergent over $z\in S_1$ on $\ov\Om$. On the other hand, it is known that the Mittag-Leffler function $E_{\al,\al}(-\la_n^\be z^\al)$ is analytic over the domain $S_2:=\{z\in\BC\mid|z|>0,|\arg z|<\pi\}$ by its definition. Thus, equation \eqref{4.19} is analytic over the domain $z\in S_1\cap S_2$ on $\ov\Om$ by the Weierstrass $M$-test. Especially, it can be extended to $\ov\Om\times(0,\infty)$, i.e.,
\begin{equation}\label{4.21}
\sum_{n=1}^\infty t^{\al-1}E_{\al,\al}\left(-\la_n^\be t^\al\right)f_{1,n}\vp_n=\sum_{n=1}^\infty t^{\al-1}E_{\al,\al}\left(-\la_n^\be t^\al\right)f_{2,n}\vp_n\quad\mbox{in }\om\times(0,\infty).
\end{equation}

Let $\{\mu_n\}_{n\in\BN}$ denote the strictly increasing sequence of distinct eigenvalues of $(-\tri)^\be$, and denote by $\{\vp_{n_k}\}_{1\le k\le m_n}$ an orthonormal basis of $\ker(\mu_n-(-\tri)^\be)$. Then we can rewrite \eqref{4.21} as
\begin{equation}\label{4.22}
\sum_{n=1}^\infty\left(\sum_{k=1}^{m_n}f_{1,n_k}\vp_{n_k}\right)a_n(t)=\sum_{n=1}^\infty\left(\sum_{k=1}^{m_n}f_{2,n_k}\vp_{n_k}\right) a_n(t)\quad\mbox{in }\om\times(0,\infty),
\end{equation}
where $a_n(t):=t^{\al-1}E_{\al,\al}(-\mu_n t^\al)$. From \eqref{4.20}, one obtains
\[
\left|\e^{-t\,\rRe\,z}\sum_{n=1}^\infty a_n(t)f_n\vp_n(x)\right|\le C\,\e^{-t\,\rRe\,z}\|f\|_{L^2(\Om)},
\]
where $\e^{-t\,\rRe\,z}$ is integrable in $(0,\infty)$ with fixed $z$ satisfying $\rRe\, z>0$. By Lebesgue's dominated convergence theorem, we can take the Laplace transform on both sides of \eqref{4.22}. By the formula (see \cite{2006Theory})
\[
\mathcal L(a_n(t))(z)=\f1{z^\al+\mu_n},\quad\rRe\,z>0,
\]
we arrive at
\[
\sum_{n=1}^\infty\left(\sum_{k=1}^{m_n}f_{1,n_k}\vp_{n_k}(x)\right)\f1{z^\al+\mu_n}=\sum_{n=1}^\infty\left(\sum_{k=1}^{m_n}f_{2,n_k}\vp_{n_k}(x)\right)\f1{z^\al+\mu_n},\quad x\in\om,\ \rRe\,z>0,
\]
which implies
\begin{equation}\label{4.25}
\sum_{n=1}^\infty\sum_{k=1}^{m_n}f_{1,n_k}\f{\vp_{n_k}(x)}{\eta+\mu_n}=\sum_{n=1}^\infty\sum_{k=1}^{m_n}f_{2,n_k}\f{\vp_{n_k}(x)}{\eta+\mu_n},\quad x\in\om,\ \eta\in S_3,
\end{equation}
where $S_3:=\{z\in\BC\mid|\arg z|<\f{\al\pi}2\}$ and $\eta=z^\al$. As
\[
\sum_{n=1}^\infty\sum_{k=1}^{m_n}\left|f_{n_k}\f1{\eta+\mu_n}\vp_{n_k}(x)\right|\le\sum_{n=1}^\infty\f{\|f_{n_k}\vp_n\|_{L^\infty(\Om)}}{|\eta+\la^\be_n|},
\]
one can see that the above series is uniformly convergent on compact subsets in $\BC\setminus\{-\mu_n\}$. Again by the Weierstrass $M$-test, both sides of \eqref{4.25} are analytic in $\BC\setminus\{-\mu_n\}$. Therefore, one can analytically continue $\eta$ such that \eqref{4.25} holds for $\eta\in\BC\setminus\{-\mu_n\}$.

Now for any $\ell\in\BN$, we can pick a sufficiently small disk centered at $-\mu_\ell$ excluding $\{-\mu_n\}_{n\ne\ell}$. Utilizing Cauchy's integral formula, we integrate \eqref{4.25} along its circle to derive
\[
u_\ell:=\sum_{k=1}^{m_\ell}(f_1-f_2,\vp_{\ell_k})\vp_{\ell_k}=0\quad\mbox{in }\om.
\]
Since $((-\tri)^\be-\mu_\ell)u_\ell=0$ in $\Om$ and $u_\ell=0$ in $\om$, the unique continuation principle of the fractional Laplacian (see \cite{2020TheCalderon}) implies $u_\ell\equiv0$ in $\Om$ for all $\ell\in\BN$. Together with the orthogonality of $\{\vp_{\ell_k}\}_{1\le k\le m_\ell}$ in $L^2(\Om)$, we obtain $(f_1-f_2,\vp_{\ell_k})=0$ for all $1\le k\le m_\ell$ and all $\ell\in\BN$. Eventually, we conclude $f_1=f_2$ in $\Om$, which completes the proof.
\end{proof}


\section{Levenberg-Marquardt method}\label{sec4}

In this section, we present the Levenberg-Marquardt method for simultaneously recovering the fractional orders, the initial value, and the spatial component of the source term from observed data. Based on Lemma~\ref{lem3.1}, we define the forward operator
\[
\cF:D(\cF)\ni(f,a,\al,\be)\longmapsto u(\,\cdot\,,\,\cdot\,;f,a,\al,\be)\in L^2(\om\times(0,T)),
\]
where $D(\cF)=L^2(\Om)\times L^2(\Om)\times(0,1)\times(0,1)$ and $u(\,\cdot\,,\,\cdot\,;f,a,\al,\be)$ denotes the solution to \eqref{eq1.1}. 

To recover the spatial component $f$, the initial value $a$, and the fractional orders $\al,\be$ from interior observations of $u$ in $\om\times(0,T)$, we employ the Levenberg-Marquardt method by considering the following minimization problem. Starting from an initial guess $(f^0,a^0,\al^0,\be^0)$ and assuming that the $k$th step approximation $(f^k,a^k,\al^k,\be^k)$ has been obtained, the $(k+1)$th step approximation $(f^{k+1}, a^{k+1},\al^{k+1},\be^{k+1})$ is determined by solving 
\begin{equation}\label{eq5.2}
\left(f^{k+1},a^{k+1},\al^{k+1},\be^{k+1}\right)=\mathop{\arg\min}_{(f,a,\al,\be)\in D(\cF)}J(f,a,\al,\be),
\end{equation}
where
\begin{align*}
J(f,a,\al,\be) & =\f12\left\|\cF\left(f^k,a^k,\al^k,\be^k\right)-u^\de+\cF'_{f^k}\left(f- f^k\right)+\cF'_{a^k}\left(a-a^k\right)\right.\\
& \quad\,\left.+\cF'_{\al^k}\left(\al-\al^k\right)+\cF'_{\be^k}\left(\be-\be^k\right)\right\|_{L^2(\om\times(0,T))}^2+\f{\xi_{k+1}}2\left\|f-f^k\right\|_{L^2(\Om)}^2\\
& \quad\,+\f{\ga_{k+1}}2\left\|a-a^k\right\|_{L^2(\Om)}^2+\f{\rho_{k+1}}2\left|\al-\al^k\right|^2+\f{\eta_{k+1}}2\left|\be-\be^k\right|^2,
\end{align*}
in which $\xi_{k+1},\ga_{k+1},\rho_{k+1},\eta_{k+1}>0$ are the regularization parameters at the $(k+1)$th step, and $\cF_{f^k}',\cF_{a^k}',\cF_{\al^k}',\cF_{\be^k}'$ denote the Fr\'echet derivatives of $\cF$ with respect to $f,a,\al,\be$ at $(f^k,a^k,\al^k,\be^k)$. Moreover, $u^\de$ is the noisy data satisfying
\[
\|u-u^\de\|_{L^2(\om\times(0,T))}\le\de.
\]
Here $u$ and $\de>0$ denote the true solution and the noise level, respectively.

In the following, we adopt a finite-dimensional approximation to solve the variational problem \eqref{eq5.2}. Let $\{\chi_n\}_{n\in\BN}$ be an appropriate basis of $L^2(\Om)$ and set
\[
f^k(x)\approx\sum_{n=1}^N c_n^k\chi_n(x),\quad a^k(x)\approx\sum_{n=1}^N d_n^k\chi_n(x),
\]
with $N\in\BN$ and expansion coefficients $c_n^k,d_n^k$ ($n=1,\dots,N$). We then define
\[
\Phi^N=\mathrm{span}\{\chi_1,\dots,\chi_N\},\quad c^k=\left(c_1^k,\dots,c_N^k\right)\in\BR^N,\quad d^k=\left(d_1^k,\dots,d_N^k\right)\in\BR^N.
\]

Next, we present the inversion algorithm for determining $f$, $a$, $\al$, and $\be$. Let
\[
b^k=\left(c^k,d^k,\al^k,\be^k\right):=\left(b_1^k,b_2^k,\dots,b_{2N+2}^k\right).
\]
Based on the above discussions, solving problem \eqref{eq5.2} can be transformed to solving the following minimization problem:
\begin{align}
\min_{b\in\BR^{2N+2}} &\left\{\f12\left\|u(\,\cdot\,,\,\cdot\,;b^k)-u^\de+\nb_{b^k}u(\,\cdot\,,\,\cdot\,;b^k)(b-b^k)^\T\right\|_{L^2(\om\times(0,T))}^2+\f{\xi_{k+1}}2\left\|f-f^k\right\|_{L^2(\Om)}^2\right.\nonumber\\
& \:\:\:\left.+\f{\ga_{k+1}}2\left\|a-a^k\right\|_{L^2(\Om)}^2+\f{\rho_{k+1}}2\left|\al-\al^k\right|^2+\f{\eta_{k+1}}2\left|\be-\be^k\right|^2\right\},\label{eq5.8}
\end{align}
where
\begin{align*}
\nb_{b^k}u(\,\cdot\,,\,\cdot\,;b^k) & =\left(\f\pa{\pa{b_1^k}}u(\,\cdot\,,\,\cdot\,;b^k),\dots,\f\pa{\pa{b_{2N+2}^k}}u(\,\cdot\,,\,\cdot\,;b^k)\right),\\
\f\pa{\pa{b_n^k}}u(\,\cdot\,,\,\cdot\,;b^k) & \approx\f{u(\,\cdot\,,\,\cdot\,;b_1^k,\dots,b_n^k+\tau,\dots,b_{2N+2}^k)-u(\,\cdot\,,\,\cdot\,;b_1^k,\dots,b_n^k,\dots,b_{2N+2}^k)}\tau.
\end{align*}
Here $\tau$ is the step size of the numerical differentiation.

The minimizer of \eqref{eq5.8} is denoted as ${b^{k+1}=(b_1^{k+1}, b_2^{k+1},\cdots, b_{2N+2}^{k+1})}$. Denoting
\begin{equation}\label{eq5.11}
b^{k+1}=b^k+\de b^k,\quad k=1,2, \dots
\end{equation}
where $\de b^k=(\de b_1^k,\dots,\de b_{2N+2}^k)$ is regarded as a perturbation of $b^k$. Thus, in order to obtain $b^{k+1}$, it suffices to compute an optimal perturbation $\de b^k$. Then the problem \eqref{eq5.8} becomes
\begin{equation}\label{eq5.12}
\min_{\de b^k\in\BR^{2N+2}}\left\{\left\|\nb_{b^k}u(\,\cdot\,,\,\cdot\,;b^k)(\de b^k)^\T-u^\de+u(\,\cdot\,,\,\cdot\,;b^k)\right\|_{L^2(\om\times(0,T))}^2+\de b^k A^k(\de b^k)^\T\right\},
\end{equation}
where
\[
A^k=\mathrm{diag}\left(\xi_{k+1}((\chi_i,\chi_j)_{L^2(\Om)})_{N\times N},\ga_{k+1}((\chi_i,\chi_j)_{L^2(\Om)})_{N\times N},\rho_{k+1},\eta_{k+1}\right).
\]
Let
\begin{align*}
Q^k & :=\left(\left(\f\pa{\pa b_i^k}u(\,\cdot\,,\,\cdot\,;b^k),\f\pa{\pa b_j^k}u(\,\cdot\,,\,\cdot\,;b^k)\right)_{L^2(\om\times(0,T))}\right)_{(2N+2)\times(2N+2)},\\
W^k & :=\left(\left(u^\de-u(\,\cdot\,,\,\cdot\,;b^k),\f\pa{\pa b_i^k}u(\,\cdot\,,\,\cdot\,;b^k)\right)_{L^2(\om\times(0,T))}\right)_{(2N+2)\times1}.
\end{align*}
We readily verify that the minimization problem \eqref{eq5.12} is reduced to solving the following normal equation
\[
(Q^k+A^k)(\de b^k)^\T=W^k.
\]
Then by the iteration procedure \eqref{eq5.11}, the optimal approximate solution can be obtained as long as arriving at the given number of iterations.


\section{Numerical experiments}\label{sec5}

In this section, we present the numerical results for some examples in the {two-dimensional} case to show the effectiveness of the Levenberg-Marquardt regularization algorithm. The noisy data is generated by adding a random perturbation, i.e.,
\[
u^\de(x,t)=(1+\de\,\mathrm{rand}(-1,1))\,u(x,t)\quad\mbox{in }\om\times(0,T),
\]
where $\mathrm{rand}(-1,1)$ denotes a random number distributed uniformly in $(-1,1)$ and the corresponding noise level is $\de$, where $\om$ is the nonempty open subset of $\Om$. To evaluate the accuracy of reconstructed solutions, we compute the relative error defined by
\[
\mathrm{err}=\f{\|f^K-f\|_{L^2(\Om)}}{\|f\|_{L^2(\Om)}}+\f{\|a^K-a\|_{L^2(\Om)}}{\|a\|_{L^2(\Om)}}+\f{|\al^K-\al|}\al+\f{|\be^K-\be|}\be
\]
with the number $K$ of iterations, where $(f^K,a^K,\al^K,\be^K)$ is regarded as the reconstructed solution produced by the Levenberg-Marquardt method. Here $(f,a,\al,\be)$ denotes the exact solution to the inverse problem. The residual $E_k$ at the $k$th iteration is given by
\[
E_k=\left\|u(\,\cdot\,,\,\cdot\,;f^k,a^k,\al^k,\be^k)-u^\de\right\|_{L^2(\om\times(0,T))}.
\]
In an iteration algorithm, usually we have to find a suitable stopping rule. In this study, we employ the well-known discrepancy principle, i.e., we choose $K$ satisfying
\[
E_K\le\si\de<E_{K-1},
\]
where $\si>1$ is a constant and can be taken heuristically to be $1.01$. If the noise level is $0$, then we take $K=20$ in the following examples. Regularization parameters $\xi_k,\ga_k,\rho_k,\eta_k$ are selected as the sigmoid-type function based on its properties given by
\begin{gather*}
\xi_k=\f1{1+\exp(\ka_0(k-k_0))},\quad\ga_k=\f1{1+\exp(\ka_1(k-k_1))},\\
\rho_k=\f1{1+\exp(\ka_2(k-k_2))},\quad\eta_k=\f1{1+\exp(\ka_3(k-k_3))},
\end{gather*}
where $k$ is the iteration step, $k_i$ and $\ka_i>0$ ($i=0,1,3$) are a priori chosen numbers and tuning parameters, respectively.

In the following examples, without loss of generality, we set
\[
p(t)=2\sin(\pi t),\quad\Om=(0,1)^2,\quad T=1.
\]
The mesh size in both space and time are chosen as $0.025$ when solving the direct problem numerically, while that in the minimization is fixed as $\tau=0.02$. For the regularization parameters above, we set
\[
\ka_0=\ka_1=\ka_2=\ka_3=0.8,\quad k_0=5,\quad k_1=4,\quad k_2=3,\quad k_3=2.
\]
We choose the approximate space
\[
\Phi^N=\mathrm{span}\left\{\sqrt2\,\sin(\pi x),\dots,\sqrt2\,\sin(N\pi x)\right\}.
\]
As for the initial guess, we simply put $f^0=a^0=0$ and $\al^0=\be^0=0.1$. The interior observation data of $u$ in $\om\times(0,T)$ is obtained by solving the direct problem \eqref{eq1.1} numerically.

Now we give two numerical examples, in which we assume that $\al=0.4$ and $\be=0.6$. The functions $f(x)$ and $a(x)$ are selected as follows.

\begin{ex}\label{ex1}
\[
f(x)=\exp\left\{-20\left((x_1-0.5)^2+(x_2-0.5)^2\right)\right\}
\]
and
\[
a(x)=\sin(\pi x_1)\sin(\pi x_2).
\]
\end{ex}

\begin{ex}\label{ex2}
\[
f(x)=\exp\left\{-30\left((x_1-0.3)^2+(x_2-0.3)^2\right)\right\}+\exp\left\{-30\left((x_1-0.7)^2+(x_2-0.7)^2\right)\right\}
\]
and
\[
a(x)=\sin(\pi x_1)\sin(\pi x_2)+\sin(2\pi x_1)\sin(2\pi x_2).
\]
\end{ex}

We first set the observation subdomain as $\om=(0.1,0.9)^2$. The inversion fractional orders and the relative errors under various noise levels are reported in Tables \ref{tab1}--\ref{tab2}, respectively, while the reconstructed spatial component $f(x)$ and initial values $a(x)$ are displayed in Figures \ref{fig1}--\ref{fig2}. The numerical results indicate that the fractional orders $\al$ and $\be$ can be recovered with satisfactory accuracy even in the presence of significant noise. In contrast, the reconstruction of the spatial component $f(x)$ and the initial value $a(x)$ proves to be more sensitive to noise. This discrepancy arises primarily from the fact that recovering the orders involves the determination of only two scalar parameters, whereas the reconstruction of the functions $f(x)$ and $a(x)$ requires the identification of multiple coefficients in their basis expansions within the inversion algorithm.  

To further assess the influence of the observation subdomain size on the reconstruction quality, we report the corresponding inversion results with different observation subdomain in Tables \ref{tab1}--\ref{tab2}. As expected, a reduction in the size of the observation area leads to a noticeable degradation in the accuracy of the recovered quantities. It is also observed that the numerical reconstructions agree well with the exact solutions when the data are contaminated by noise levels up to $4\%$. Overall, these results demonstrate that the proposed algorithm performs reliably and exhibits good stability.

\begin{figure}[htbp]\centering
\subfigure[True solution $f$]{\includegraphics[width=0.32\textwidth]{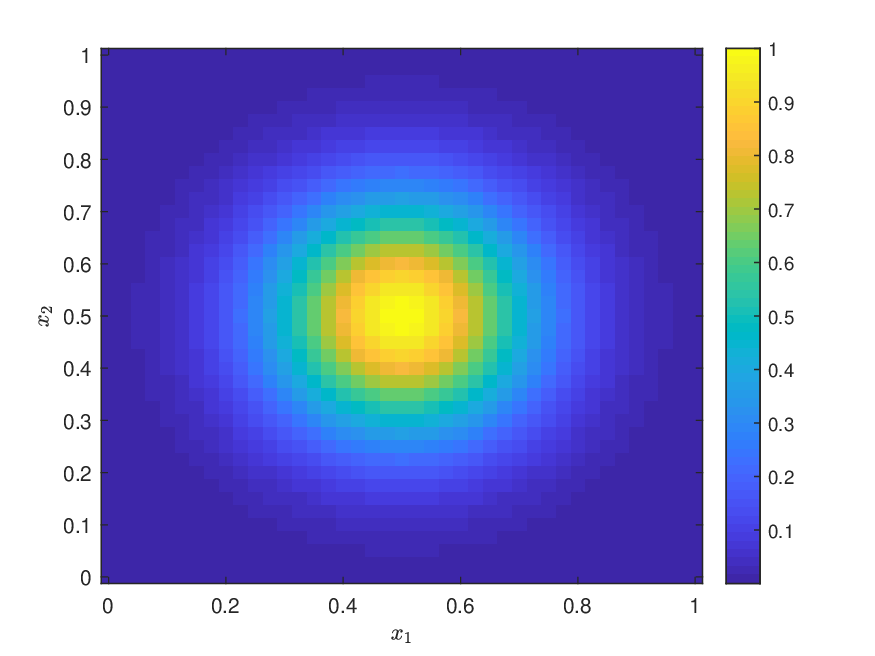}}
\subfigure[Reconstructed solution $f^K$]{\includegraphics[width=0.32\textwidth]{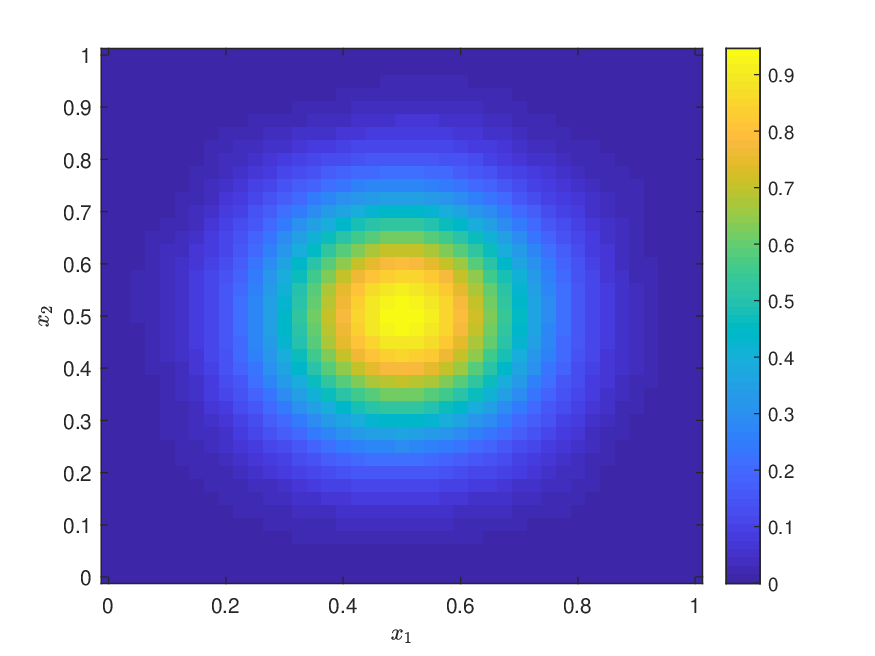}}
\subfigure[Absolute error $|f-f^K|$]{\includegraphics[width=0.32\textwidth]{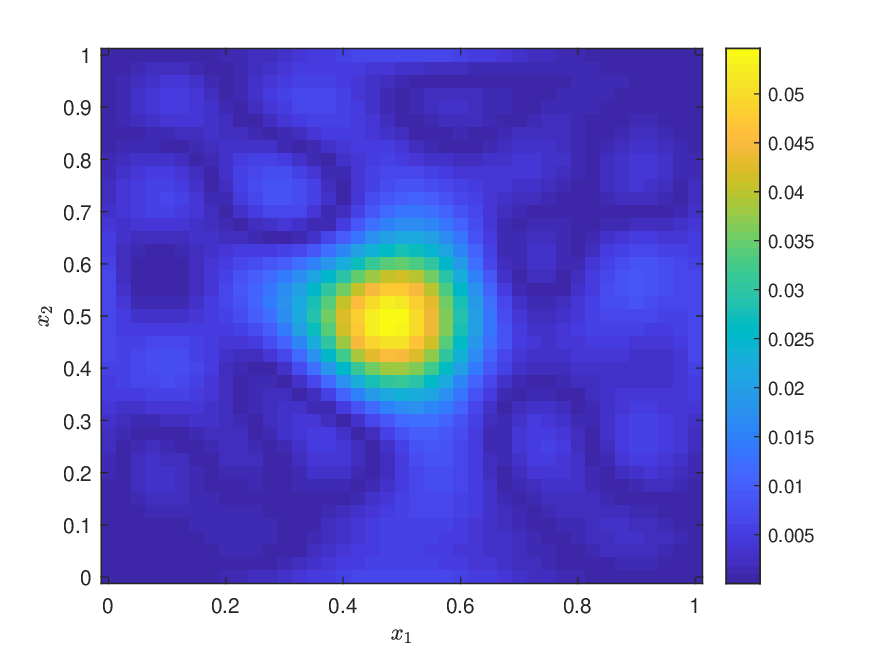}}
\subfigure[True solution $a$]{\includegraphics[width=0.32\textwidth]{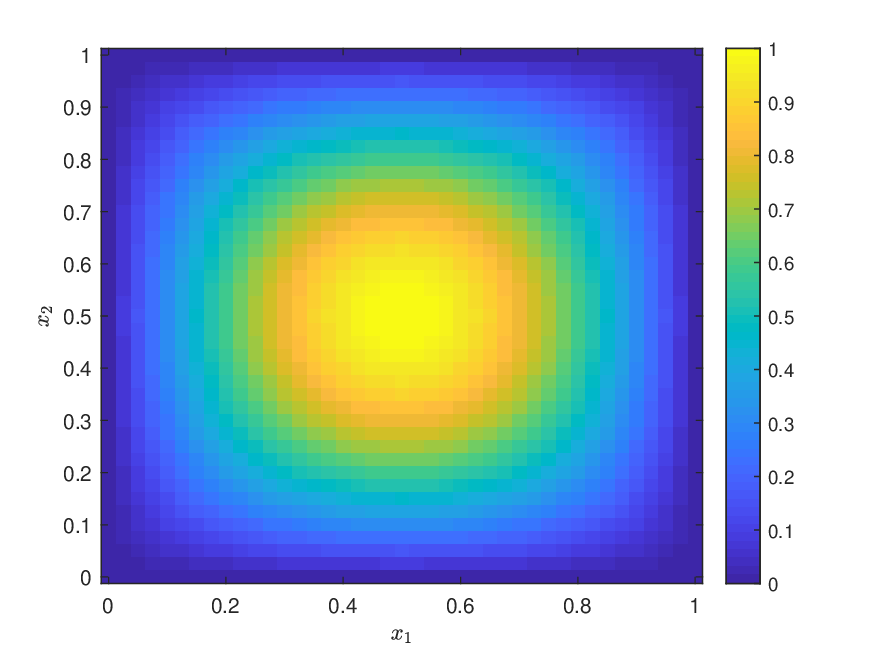}}
\subfigure[Reconstructed solution $a^K$]{\includegraphics[width=0.32\textwidth]{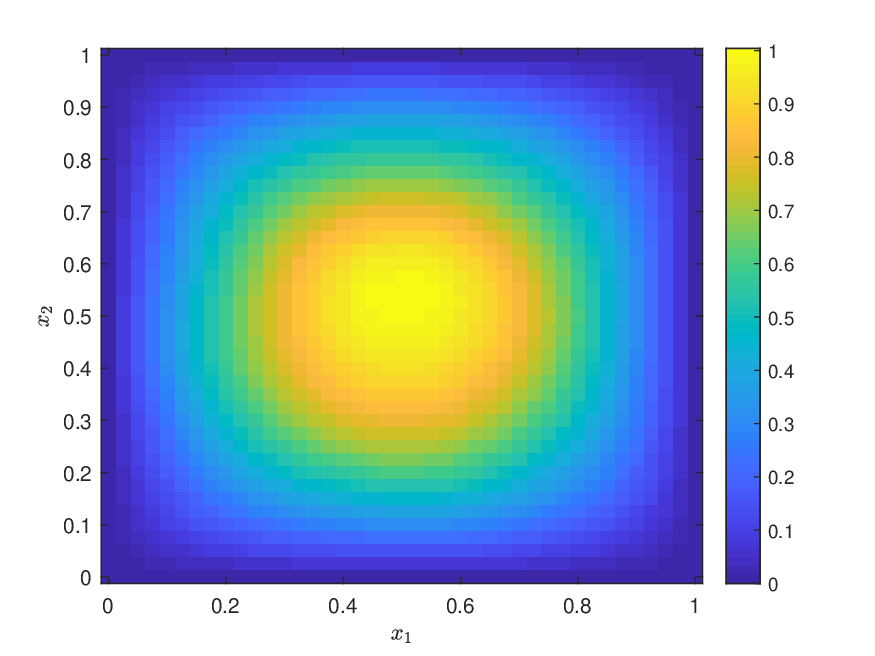}}
\subfigure[Absolute error $|a-a^K|$]{\includegraphics[width=0.32\textwidth]{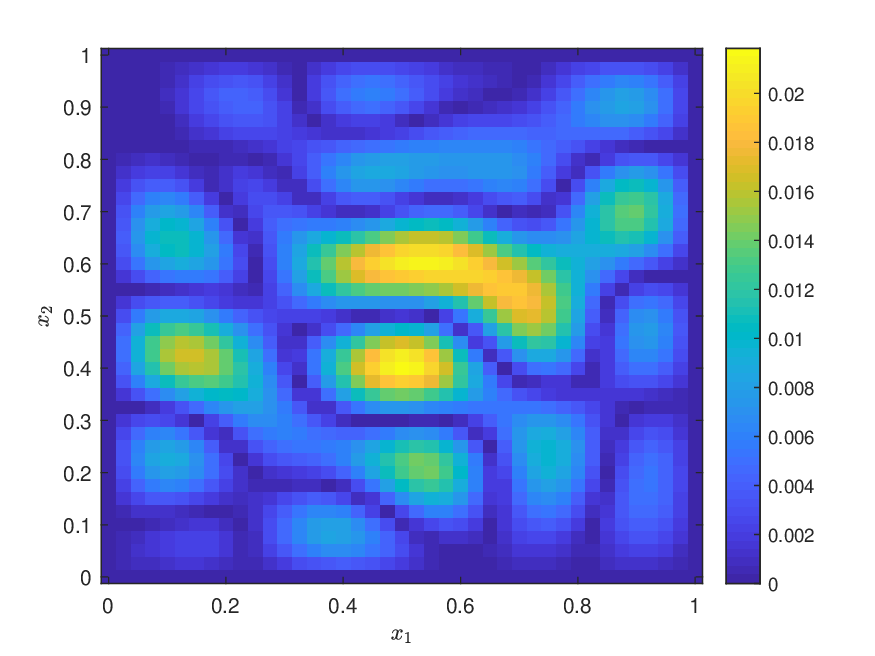}}
\caption{True solution (left), reconstructed solution (middle), and absolute error (right) of Example \ref{ex1} for $\om=(0.1,0.9)^2$ and $\de= 0.04$.}\label{fig1}
\end{figure}

\begin{figure}[htbp]\centering
\subfigure[True solution $f$]{\includegraphics[width=0.32\textwidth]{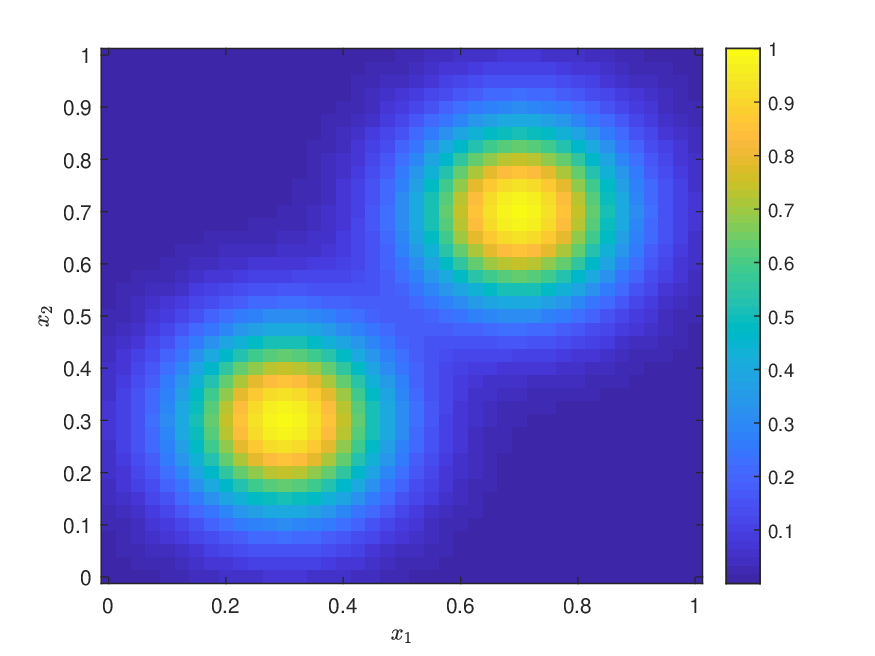}}
\subfigure[Reconstructed solution $f^K$]{\includegraphics[width=0.32\textwidth]{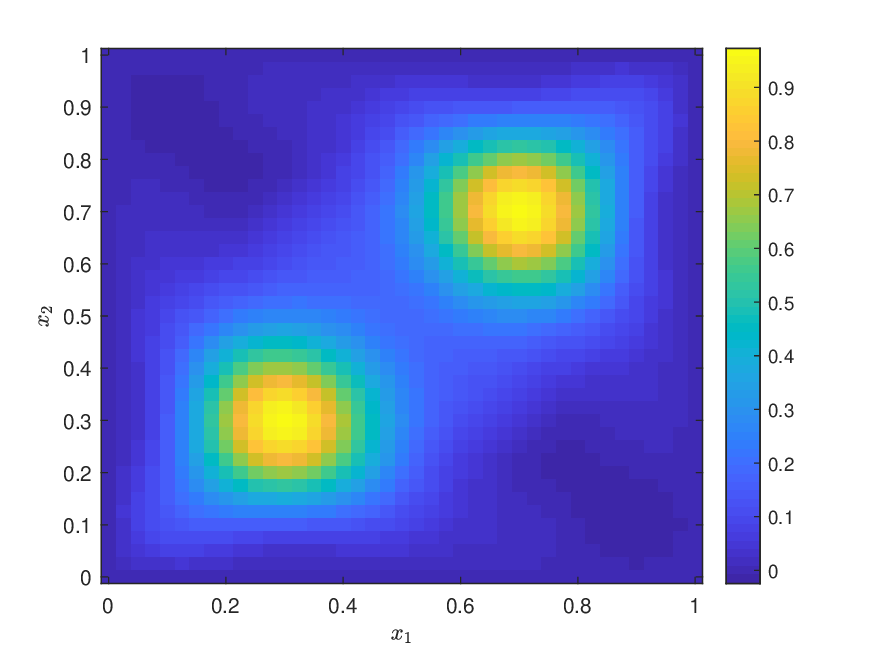}}
\subfigure[Absolute error $|f-f^K|$]{\includegraphics[width=0.32\textwidth]{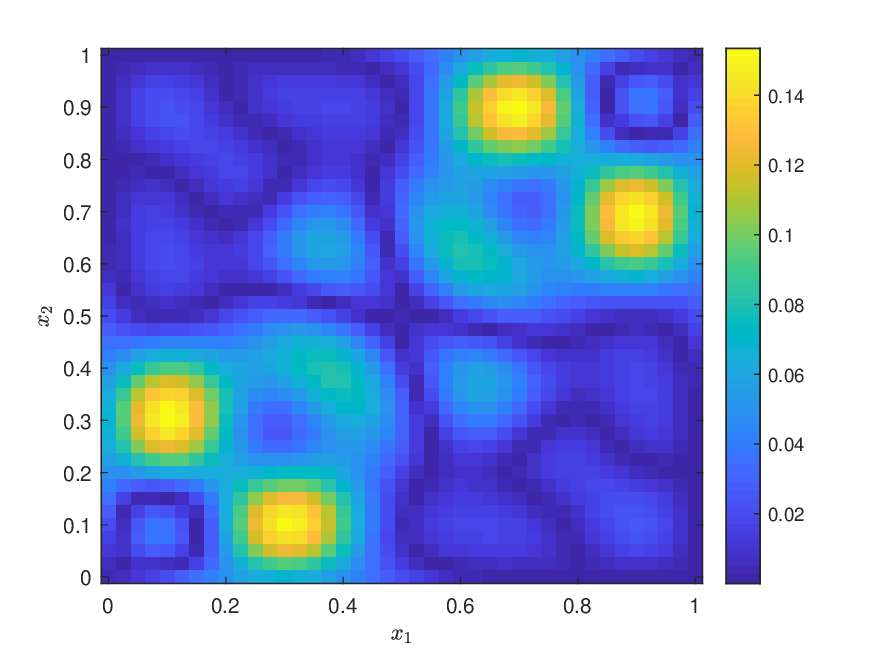}}
\subfigure[True solution $a$]{\includegraphics[width=0.32\textwidth]{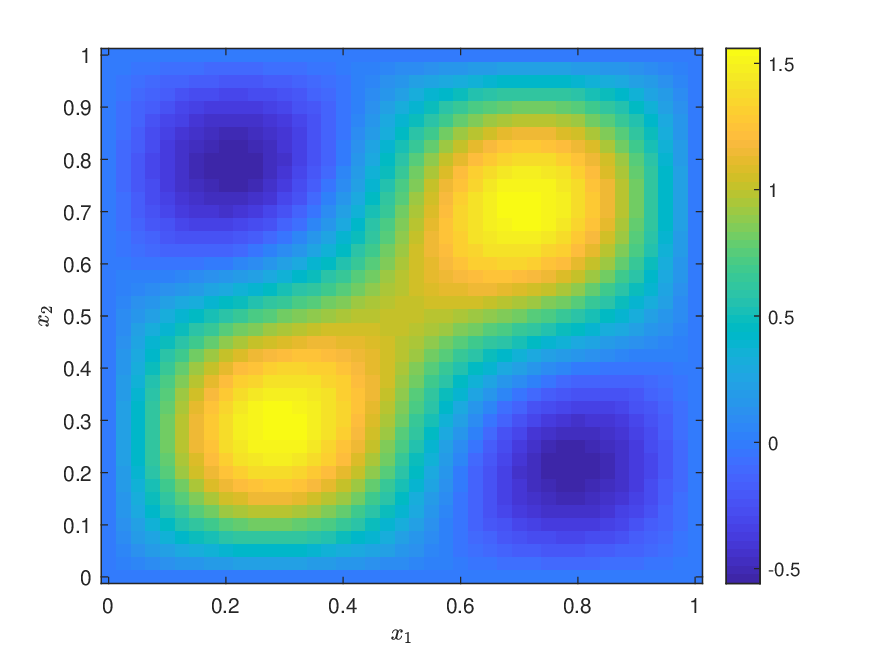}}
\subfigure[Reconstructed solution $a^K$]{\includegraphics[width=0.32\textwidth]{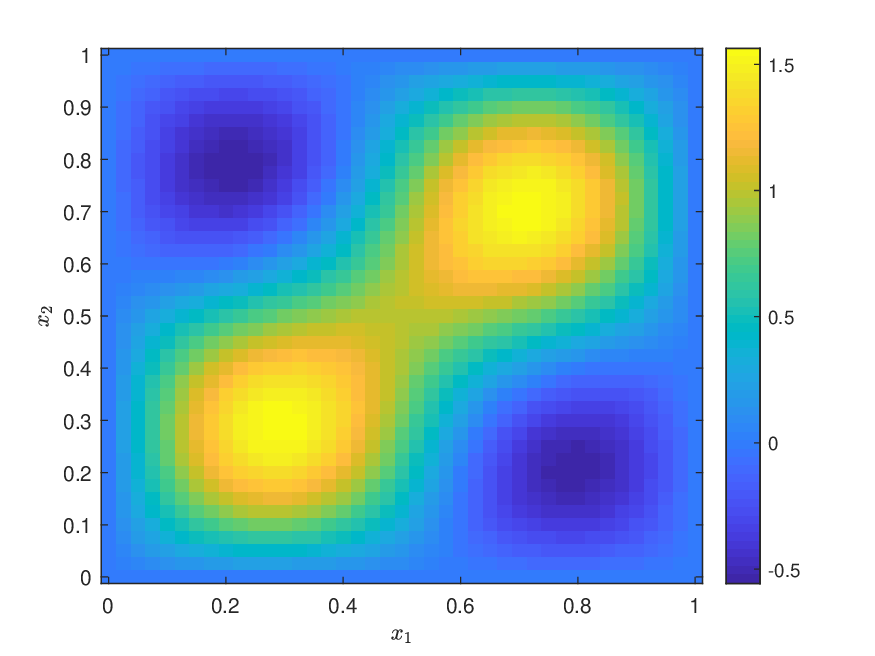}}
\subfigure[Absolute error $|a-a^K|$]{\includegraphics[width=0.32\textwidth]{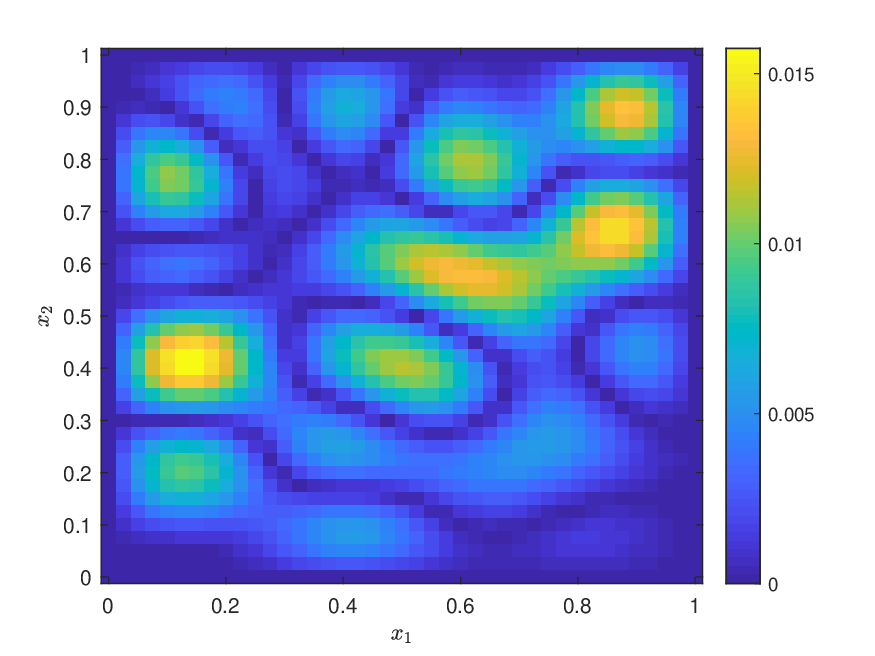}}
\caption{True solution (left), reconstructed solution (middle), and absolute error (right) of Example \ref{ex2} for $\om=(0.1,0.9)^2$ and $\de=0.04$.}\label{fig2}
\end{figure}

\begin{table}[htbp]\centering
\caption{The reconstructed fractional orders and relative errors of Example \ref{ex1} for various noise levels and observation subdomain.\label{tab1}}
\begin{tabular}{cc|cccc}
\toprule 
$\om$ & $\de$ & $K$ & $\al^K$ & $\be^K$ & err\\
\midrule
$(0.1,0.9)^2$ & $0.00$ & 10 & $0.4000$ & $0.6001$ & $0.0529$\\
$(0.1,0.9)^2$ & $0.01$ & 10 & $0.3998$ & $0.5998$ & $0.0568$\\
$(0.1,0.9)^2$ & $0.02$ & 9 & $0.3995$ & $0.5991$ & $0.0662$\\
$(0.1,0.9)^2$ & $0.04$ & 8 & $0.3987$ & $0.5983$ & $0.0824$\\
\midrule
$(0.2,0.8)^2$ & $0.00$ & 11 & $0.3980$ & $0.5965$ & $0.1022$\\
$(0.2,0.8)^2$ & $0.01$ & 11 & $0.3977$ & $0.5963$ & $0.1078$\\
$(0.2,0.8)^2$ & $0.02$ & 9 & $0.3968$ & $0.5954$ & $0.1276$\\
$(0.2,0.8)^2$ & $0.04$ & 8 & $0.3951$ & $0.5936$ & $0.1667$\\
\bottomrule
\end{tabular}
\end{table}

\begin{table}[htbp]\centering
\caption{The reconstructed fractional orders and relative errors of Example \ref{ex2} for various noise levels and observation subdomain.\label{tab2}}
\begin{tabular}{cc|cccc}
\toprule 
$\om$ & $\de$ & $K$ & $\al^K$ & $\be^K$ & err\\
\midrule
$(0.1,0.9)^2$ & $0.00$ & 9 & $0.3996$ & $0.5993$ & $0.1149$\\
$(0.1,0.9)^2$ & $0.01$ & 8 & $0.3991$ & $0.5988$ & $0.1197$\\
$(0.1,0.9)^2$ & $0.02$ & 8 & $0.3965$ & $0.5943$ & $0.1354$\\
$(0.1,0.9)^2$ & $0.04$ & 7 & $0.3943$ & $0.5926$ & $0.1642$\\
\midrule
$(0.2,0.8)^2$ & $0.00$ & 10 & $0.3986$ & $0.5967$ & $0.2071$\\
$(0.2,0.8)^2$ & $0.01$ & 9 & $0.3984$ & $0.5964$ & $0.2187$\\
$(0.2,0.8)^2$ & $0.02$ & 8 & $0.3978$ & $0.5953$ & $0.2304$\\
$(0.2,0.8)^2$ & $0.04$ & 7 & $0.3961$ & $0.5921$ & $0.2743$\\
\bottomrule
\end{tabular}
\end{table}


\section{Concluding remarks}\label{sec6}

This paper addresses the simultaneous recovery of the spatial component of the source term, the initial value, and the fractional orders in space and time for nonlocal diffusion equations from internal measurements. The uniqueness of the corresponding inverse problem is established using the asymptotic behavior of solutions, analytic continuation, the Laplace transform, and properties of analytic functions. The Levenberg-Marquardt method is employed for the numerical reconstruction and, through numerical experiments, is shown to be both effective and stable, thereby providing a promising approach for solving inverse problems in nonlocal diffusion equations.


\section*{Declaration}

\subsection*{Conflict of interest}

The authors declare that there are no conflicts of interest.

\subsection*{Funding statement}

This work was partially supported by the National Natural Science Foundation of China (No.\! 12271277) and the Open Research Fund of Key Laboratory of Nonlinear Analysis \& Applications (Central China Normal University), Ministry of Education, China. The second author also thanks Ningbo Youth Leading Talent Project (No.\! 2024QL045). The third author is supported by JSPS KAKENHI Grant Numbers JP23KK0049 and JP26K06926, Guangdong Basic and Applied Basic Research Foundation (No.\! 2025A1515012248) and FY2025 MUSUBIME of Kyoto University.

\subsection*{Author Contributions}

All authors contribute equally to this manuscript.

\subsection*{Acknowledgments}

The authors thank the anonymous referee for valuable comments. This work was partially supported by the National Natural Science Foundation of China (No.\! 12271277) and the Open Research Fund of Key Laboratory of Nonlinear Analysis \& Applications (Central China Normal University), Ministry of Education, China. The second author also thanks Ningbo Youth Leading Talent Project (No.\! 2024QL045). The third author is supported by JSPS KAKENHI Grant Numbers JP23KK0049 and JP26K06926, Guangdong Basic and Applied Basic Research Foundation (No.\! 2025A1515012248) and FY2025 MUSUBIME of Kyoto University.


\bibliographystyle{unsrt}

\begin{thebibliography}{35}

\bibitem{AF03}
R.A. Adams, {\it Sobolev Spaces}, Academic Press, New York, 1975.

\bibitem{2018Inverse}
M. Ali, S. Aziz and S. Malik, Inverse source problem for a space-time fractional diffusion equation, {\it Fract. Calc. Appl. Anal.}, {\bf21} (2018), 844--863.

\bibitem{2023RecoveryCen}
S. Cen, B. Jin, Y. Liu and Z. Zhou, Recovery of multiple parameters in subdiffusion from one lateral boundary measurement, {\it Inverse Problems}, {\bf39} (2023), Paper No. 104001, 31 pp.

\bibitem{1995Introduction}
G. B. Folland, {\it Introduction to Partial Differential Equations}, Princeton University Press, Princeton, 1995.

\bibitem{2020TheCalderon}
T. Ghosh, M. Salo and G. Uhlmann, The Calder\'on problem for the fractional Schr\"odinger equation, {\it Anal. \& PDE}, {\bf13} (2020), 455--475.

\bibitem{GL15}
R. Gorenflo, Y. Luchko and M. Yamamoto, Time-fractional diffusion equation in the fractional Sobolev spaces, {\it Fract. Calc. Appl. Anal.}, {\bf18} (2015), 799--820.

\bibitem{2021SimultaneousGuerngar}
N. Guerngar, E. Nane, R. Tinaztepe, S. Ulusoy and H.W. Van Wyk, Simultaneous inversion for the fractional exponents in the space-time fractional diffusion equation $\pa_t^\be u=-(-\tri)^{\al/2}u-(-\tri)^{\ga/2}u$, {\it Fract. Calc. Appl. Anal.}, {\bf24} (2021), 818--847.

\bibitem{2020Inverse problems}
T. Helin, M. Lassas, L. Ylinen and Z. Zhang, Inverse problems for heat equation and space-time fractional diffusion equation with one measurement, {\it J. Differential Equations}, {\bf269} (2020), 7498--7528.

\bibitem{2020Determination}
J. Janno, Determination of time-dependent sources and parameters of nonlocal diffusion and wave equations from final data, {\it Fract. Calc. Appl. Anal.}, {\bf23} (2020), 1678--1701.

\bibitem{2018Backward}
J. Jia, J. Peng, J. Gao and Y. Li, Backward problem for a time-space fractional diffusion equation, {\it Inverse Probl. Imaging}, {\bf12} (2018), 773--799.

\bibitem{2017Harnack's}
J. Jia, J. Peng and J. Yang, Harnack's inequality for a space-time fractional diffusion equation and applications to an inverse source problem, {\it J. Differential Equations}, {\bf262} (2017), 4415--4450.

\bibitem{2021The}
Y. Kian, Z. Li, Y. Liu and M. Yamamoto, The uniqueness of inverse problems for a fractional equation with a single measurement, {\it Math. Ann.}, {\bf380} (2021), 1465--1495.

\bibitem{2006Theory}
A.A. Kilbas, H.M. Srivastava and J.J. Trujillo, {\it Theory and Applications of Fractional Differential Equations}, Elsevier, Amsterdam, 2006.

\bibitem{2020The Calderon}
R. Lai, Y. Lin and A. R\"uland, The Calder\'on problem for a space-time fractional parabolic equation, {\it SIAM J. Math. Anal.}, {\bf52} (2020), 2655--2688.

\bibitem{2018An}
Y. Li and T. Wei, An inverse time-dependent source problem for a time-space fractional diffusion equation, {\it Appl. Math. Comput.}, {\bf336} (2018), 257--271.

\bibitem{2019Inverse problems li}
Z. Li, Y. Liu and M. Yamamoto, Inverse problems of determining parameters of the fractional partial differential equations, {\it Handbook of Fractional Calculus with Applications, Volume 2: Fractional Differential Equations}, De Gruyter, Berlin (2019), 431--442.

\bibitem{2019Inverse li}
Z. Li and M. Yamamoto, Inverse problems of determining coefficients of the fractional partial differential equations, {\it Handbook of Fractional Calculus with Applications, Volume 2: Fractional Differential Equations}, De Gruyter, Berlin (2019), 443--464.

\bibitem{2019Inverse liu}
Y. Liu, Z. Li and M. Yamamoto, Inverse problems of determining sources of the fractional partial differential equations, {\it Handbook of Fractional Calculus with Applications, Volume 2: Fractional Differential Equations}, De Gruyter, Berlin (2019), 411--430.

\bibitem{2022Uniqueness}
Y. Liu and M. Yamamoto, Uniqueness of orders and parameters in multi-term time-fractional diffusion equations by short-time behavior, {\it Inverse Problems}, {\bf39} (2022), Paper No. 024003, 28 pp.

\bibitem{1993Linear pa}
L. Rodino, {\it Linear Partial Differential Operators in Gevrey Spaces}, World Scientific, Singapore, 1993.

\bibitem{2024Recovering}
K. Lyu and H. Cheng, Recovering a space-dependent source term for distributed order time-space fractional diffusion equation, {\it Numer. Algorithms}, {\bf99} (2024), 1317--1342.

\bibitem{2021SimultaneousMalik}
S. Malik, A. Ilyas and A. Samreen, Simultaneous determination of a source term and diffusion concentration for a multi-term space-time fractional diffusion equation, {\it Math. Model. Anal.}, {\bf26} (2021), 411--431.

\bibitem{2000The}
R. Metzler and J. Klafter, The random walk's guide to anomalous diffusion: a fractional dynamics approach, {\it Phys. Rep.}, {\bf339} (2000), 1--77.

\bibitem{1998Fractional}
I. Podlubny, {\it Fractional Differential Equations}, Academic Press, San Diego, 1999.

\bibitem{2002Waiting-times}
M. Raberto, E. Scalas and F. Mainardi, Waiting-times and returns in high-frequency financial data: an empirical study, {\it Phys. A}, {\bf314} (2002), 749--755.

\bibitem{2021Direct}
J. Restrepo and D. Suragan, Direct and inverse Cauchy problems for generalized space-time fractional differential equations, {\it Adv. Differential Equations}, {\bf26} (2021), 305--339.

\bibitem{2011Initial}
K. Sakamoto and M. Yamamoto, Initial value/boundary value problems for fractional diffusion-wave equations and applications to some inverse problems, {\it J. Math. Anal. Appl.}, {\bf382} (2011), 426--447.

\bibitem{2021Simultaneous}
L. Sun, Y. Li and Y. Zhang, Simultaneous inversion of the potential term and the fractional orders in a multi-term time-fractional diffusion equation, {\it Inverse Problems}, {\bf37} (2021), Paper No. 055007, 26 pp.

\bibitem{2016Simultaneous}
S. Tatar, R. T\i naztepe and S. Ulusoy, Simultaneous inversion for the exponents of the fractional time and space derivatives in the space-time fractional diffusion equation, {\it Appl. Anal.}, {\bf95} (2016), 1--23.

\bibitem{2015An inverse}
S. Tatar and S. Ulusoy, An inverse source problem for a one-dimensional space-time fractional diffusion equation, {\it Appl. Anal.}, {\bf94} (2015), 2233--2244.

\bibitem{2023The quasi-reversibility}
N.V. Duc, N.V. Thang and N.T. Th\`anh, The quasi-reversibility method for an inverse source problem for time-space fractional parabolic equations, {\it J. Differential Equations}, {\bf344} (2023), 102--130.

\bibitem{2021Unknown}
F. Yang, Q. Wang and X. Li, Unknown source identification problem for space-time fractional diffusion equation: optimal error bound analysis and regularization method, {\it Inverse Probl. Sci. Eng.}, {\bf29} (2021), 2040--2084.

\bibitem{2025Inverse source}
K. Yu, Z. Li and Y. Liu, Inverse source problem with a posteriori interior measurements for space-time fractional diffusion equations, {\it Comput. Math. Appl.}, {\bf200} (2025), 305--314.

\bibitem{2002Subdiffusion-limited}
S. Yuste and K. Lindenberg, Subdiffusion-limited reactions, {\it Chem. Phys.}, {\bf284} (2002), 169--180.

\bibitem{2018Bayesian}
Y. Zhang, J. Jia, L. Yan, Bayesian approach to a nonlinear inverse problem for a time-space fractional diffusion equation, {\it Inverse Problems}, {\bf34}(12) (2018): 125002.

\end{thebibliography}

\end{document}